\documentclass[12pt]{amsart}

\newtheorem{theorem}{Theorem}[section]
\newtheorem{lemma}[theorem]{Lemma}
\newtheorem{corollary}[theorem]{Corollary}

\newtheorem{conjecture}[theorem]{Conjecture}

\theoremstyle{plain}

\newtheorem{remark}[theorem]{Remark}

\newtheorem{question}[theorem]{Question}

\theoremstyle{definition}

\theoremstyle{remark}

\numberwithin{equation}{section}

\usepackage{CJK}
\usepackage{fancyhdr}
\usepackage{amscd}
\usepackage{amsmath}
\usepackage{amsthm}
\usepackage{amssymb}
\begin{document}

\title[]{On the image of the Abel-Jacobi map}

\author{Sen Yang}
\address{Shing-Tung Yau Center of Southeast University \\ 
Southeast University \\
Nanjing, China\\
}
\address{School of Mathematics \\  Southeast University \\
Nanjing, China\\
}
\address{Yau Mathematical Sciences Center \\
Tsinghua University \\
Beijing, China\\
}
\email{101012424@seu.edu.cn; syang@math.tsinghua.edu.cn}

\date{}

\maketitle

\begin{abstract}
We prove and generalize an observation of Green and Griffiths on the infinitesimal form of the Abel-Jacobi map. As an application, we prove that the infinitesimal form of a conjecture by Griffiths and Harris \cite{GH} is true.
\end{abstract}

\tableofcontents

\section{Introduction}
\label{Introduction}

In \cite{GH}, Griffiths and Harris conjectured that 
\begin{conjecture}[\cite{GH}] \label{conjecture: GH}
Let $X \subset \mathbb{P}_{\mathbb{C}}^{4}$ be a general hypersurface of degree $d \geq 6$, we use 
\[
\psi: CH_{hom}^{2}(X) \to J^{2}(X)
\]
to denote the Abel-Jacobi map from algebraic 1-cycles on $X$ homologically equivalent to zero to the intermediate Jacobian $J^{2}(X)$, $\psi$ is zero.
\end{conjecture}

There are many known results showing that the image of the Abel-Jacobi map is torsion, or even that a Chow group is torsion. We give an illustration of the use of tangent spaces to Chow groups to prove infinitesimal versions of these results. For example, Green and Voisin studied this conjecture and showed that
\begin{theorem} [Theorem 0.1 of \cite{Green-JDG} \footnote{In fact, Green
 \cite{Green-JDG} proved analogous results in all dimension.}, 1.5 of \cite{Voisin}]
For $X \subset \mathbb{P}_{\mathbb{C}}^{4}$ a general hypersurface of degree $d \geq 6$, 
the image of the Abel-Jacobi map on algebraic 1-cycles on $X$ homologically equivalent to zero has image contained in the torsion points of the intermediate Jacobian.
\end{theorem}

In this note, we prove and generalize an observation of Green-Griffiths (Lemma \ref{lemma: GG}) on the infinitesimal form of Abel-Jacobi map, see Corollary \ref{corollary: inf'l AJ}. Moreover, we prove that the infinitesimal form of Conjecture \ref{conjecture: GH} is true, see Theorem \ref{Theorem: main}. 

\textbf{Notation.}
In Conjecture \ref{conjecture: GH}, {\it  general} means {\it outside a countable union of proper subvarieties}. In Theorem \ref{Theorem: main}, {\it  general} means {\it all the coefficients of the defining equation of $X$ are algebraically independent}.

\section{Main results}
Let $X$ be a smooth projective variety over $\mathbb{C}$. For each positive integer $p$,
 Green and Griffiths  \cite{GG-2} study the tangent space to the cycle group $Z^{p}(X)$. In the last chapter of  \cite{GG-2}, among other open questions, they ask
\begin{question} [(10.3) on page 186  \cite{GG-2}] \label{question: GG}
Can one define the Bloch-Gersten-Quillen sequence $\mathcal{G}_j$ on infinitesimal neighborhoods $X_j$ of $X$
so that
\[   \label{equGGquestion}
 ker(\mathcal{G}_1 \to \mathcal{G}_0) =  \underline{\underline{T}}\mathcal{G}_0?
\]
\end{question}
Let's briefly explain the notations of this question.  $X_{j}=X \times_{\mathbb{C}} \mathrm{Spec}(\mathbb{C}[t]/(t^{j+1}))$ denotes the $j$-th order infinitesimal thickening of a smooth projective variety $X$ over $\mathbb{C}$, and $\mathcal{G}_0$ denotes the ordinary Bloch-Gersten-Quillen resolution, i.e., the flasque resolution of the $K$-theory sheaf  $K_{p}(O_{X})$ for some $p$.  $\underline{\underline{T}}\mathcal{G}_0$ denotes the ``tangent sequence" of $\mathcal{G}_0$,  which is the Cousin resolution of  absolute differentials.

To give a concrete example, to study the first order deformation of 0-cycles on a surface $X$, let $X_{1}=X \times \mathrm{Spec}(\mathbb{C}[t]/(t^2))$, Green and Griffiths point out that $\mathcal{G}_1$ should be something like \footnote{See (10.5) on page 186  \cite{GG-2}.}
\begin{equation}
0 \to K_{2}(O_{X_{1}}) \to K_{2}(\mathbb{C}(X_{1})) \to \bigoplus_{Y_{1}}\mathbb{C}(Y_{1}) \to \bigoplus_{x_{1}}\mathbb{Z}_{x_{1}} \to 0,
\end{equation}
which gives a flasque resolution of  $K_{2}(O_{X_{1}})$.

Question \ref{question: GG} has been solved in \cite{DHY}. In particular, (2.1) has the form
\begin{align*}
(2.2) \  \ 0 \to & K_{2}(O_{X_{1}}) \to K_{2}(\mathbb{C}(X_{1}))  \to 
 \bigoplus_{y \in X^{(1)}_{1}}K_{1}(O_{X_{1},y} \ \mathrm{on} \ y ) \\
 & \to \bigoplus_{x \in X_{1}^{(2)}}K_{0}(O_{X_{1},x} \ \mathrm{on} \ x) \to 0,
\end{align*}
where K-groups are Thomason-Trobaugh K-groups.

To give one more example,  to study the first order deformation of 1-cycles on a smooth projective three-fold $X$, $\mathcal{G}_1$ has the form
{\small
\begin{align*}
(2.3)  \  \  0 \to K_{2}(O_{X_{1}}) & \to K_{2}(\mathbb{C}(X_{1})) \to \bigoplus_{z \in X^{(1)}_{1}}K_{1}(O_{X_{1},z} \ \mathrm{on} \ z )  \\
& \to \bigoplus_{y \in X_{1}^{(2)}}K_{0}(O_{X_{1},y} \ \mathrm{on} \ y) 
 \to \bigoplus_{x \in X_{1}^{(3)}}K_{-1}(O_{X_{1},x} \ \mathrm{on} \ x) \to 0.
\end{align*}

Green and Griffiths observe that 
\begin{lemma}  [page 189(line 21-27) of  \cite{GG-2}] \label{lemma: GG}
For $X$ a smooth projective three-fold over $\mathbb{C}$, the infinitesimal form of the Abel-Jacobi map
\[
\psi: CH_{hom}^{2}(X) \to J^{2}(X),
\]
 is given by 
\[
\delta \psi: H^{2}(\Omega_{X/ \mathbb{Q}}^{1}) \to H^{2}(\Omega_{X/ \mathbb{C}}^{1}).
\]
\end{lemma}

Green and Griffiths use the sequence (2.1) (or (2.2)), which is to study the first order deformation of 0-cycles on a surface, to derive this Lemma. To study the first order deformation of 1-cycles on a three-fold, one should use the sequence (2.3), though both the sequence (2.1) and (2.3) study codimension 2 cycles. From the K-theoretic viewpoint, there is no negative K-groups in the sequence (2.1) (or (2.2)) because of dimensional reason, while the negative K-group $K_{-1}(O_{X_{1},x} \ \mathrm{on} \ x) $ appears in the sequence (2.3).

We will need a generalization of Lemma \ref{lemma: GG} in what follows. For this purpose, we recall basic properties about Deligne cohomology, and its connection to the Abel-Jacobi map.

Let $X$ be a smooth projective variety over $\mathbb{C}$, for each positive integer $p$,
there exists the commutative diagram:
\[
  \begin{CD}
    0  @>>> CH^{p}_{hom}(X) @>>>  CH^{p}(X) @>>> \dfrac{CH^{p}(X)}{CH^{p}_{hom}(X)} @>>> 0 \\
     @.  @V \psi VV  @V r VV   @V cl VV   @. \\
     0  @>>> J^{p}(X) @>>> H_{\mathcal{D}}^{2p}(X, \mathbb{Z}(p))  @>>> Hg^{p}(X) @>>> 0 . \\
  \end{CD}
\]
We briefly explain the terminologies in this diagram,
\begin{itemize}
\item $J^{p}(X) $ is Griffiths intermediate Jacobian and $\psi$ is the Abel-Jacobi map.  

\item $H_{\mathcal{D}}^{2p}(X, \mathbb{Z}(p))$ is Deligne cohomology and $r$ is the Deligne cycle class map. For the construction of $r$, we refer to \cite{El-Z, EV, Jannsen}.

\item $Hg^{p}(X)$ is the Hodge group, defined as $f^{-1}(H^{p,p}(X))$, where $f: H^{2p}(X, \mathbb{Z}(p))  \to H^{2p}(X, \mathbb{C})$ is induced by the inclusion $\mathbb{Z} \to \mathbb{C}$. The map $cl$ is the cycle class map, see \cite{L} for a survey.

\end{itemize}

 Since the Hodge group $Hg^{p}(X)$ ($\subseteq H^{2p}(X, \mathbb{Z}(p))$ is discrete, so $J^{p}(X) $ and $H_{\mathcal{D}}^{2p}(X, \mathbb{Z}(p)) $ have the same tangent space. Similarly, $\dfrac{CH^{p}(X)}{CH^{p}_{hom}(X)} $ ($\subseteq H^{2p}(X, \mathbb{Z})$) is discrete so that $CH^{p}(X)$ and $CH^{p}_{hom}(X)$ have the same tangent space. So we have,
 \begin{theorem}[ known to experts, e.g., see \cite{Voisin} (line 9-11 on page 707) ] \label{theorem: AJagreeDel}
Let $X$ be a smooth projective variety over $\mathbb{C}$, for each positive integer $p$, the infinitesimal form of the Abel-Jacobi map,
\[
\psi: CH_{hom}^{p}(X) \to J^{p}(X),
\]
agrees with that of the  Deligne cycles class map, 
\[
r: CH^{p}(X) \to H_{\mathcal{D}}^{2p}(X, \mathbb{Z}(p)).
\]
\end{theorem}

\begin{theorem} [Theorem 2.11 \cite{Yang}]
Let $X$ be a smooth projective variety over $\mathbb{C}$, for each positive integer $p$, the infinitesimal form of  the Deligne cycle class map
\[
r: CH^{p}(X) \to H_{\mathcal{D}}^{2p}(X, \mathbb{Z}(p)),
\]
is given by
\[
\delta r: H^{p}(\Omega_{X/ \mathbb{Q}}^{p-1}) \to H^{p}(\Omega_{X/ \mathbb{C}}^{p-1}),
\]
where $\delta r$ is induced by the natural map $\Omega_{X/ \mathbb{Q}}^{p-1} \to \Omega_{X/ \mathbb{C}}^{p-1}$.
\end{theorem}

By Theorem \ref{theorem: AJagreeDel}, one has,
\begin{corollary} \label{corollary: inf'l AJ}
Let $X$ be a smooth projective variety over $\mathbb{C}$, for each positive integer $p$, the infinitesimal form of the Abel-Jacobi map,
\[
\psi: CH_{hom}^{p}(X) \to J^{p}(X)
\]
is given by
\begin{equation}
\delta r: H^{P}(\Omega_{X/ \mathbb{Q}}^{p-1}) \to H^{p}(\Omega_{X/ \mathbb{C}}^{p-1}),
\end{equation}
where $\delta r$ is induced by the natural map $\Omega_{X/ \mathbb{Q}}^{p-1} \to \Omega_{X/ \mathbb{C}}^{p-1}$.
\end{corollary}

In particular, for $X$ a smooth projective three-fold over $\mathbb{C}$, the infinitesimal form of the Abel-Jacobi map
\[
\psi: CH_{hom}^{2}(X) \to J^{2}(X)
\]
 is given by 
\[
\delta r: H^{2}(\Omega_{X/ \mathbb{Q}}^{1}) \to H^{2}(\Omega_{X/ \mathbb{C}}^{1}).
\]
This proves Lemma \ref{lemma: GG}.

\begin{remark}
The infinitesimal form of Abel-Jacobi map (2.2) does not agree with the infinitesimal Abel-Jacobi map in \cite{Green-note}.
\end{remark}

Now, we prove that the infinitesimal form of Conjecture \ref{conjecture: GH} is true,
\begin{theorem} \label{Theorem: main}
For $X \subset \mathbb{P}_{\mathbb{C}}^{4}$ a general hypersurface of degree $d \geq 6$, the image of the infinitesimal form of the Abel-Jacobi map
\[
\delta r: H^{2}(\Omega_{X/ \mathbb{Q}}^{1}) \to H^{2}(\Omega_{X/ \mathbb{C}}^{1})
\]
 is zero.
\end{theorem}

\begin{proof}
There is a natural short exact sequence of sheaves
\[
0 \to \Omega_{\mathbb{C}/\mathbb{Q}}^{1}\otimes_{\mathbb{C}} O_{X} \to \Omega_{X/\mathbb{Q}}^{1} \to \Omega_{X/\mathbb{C}}^{1} \to 0.
\]

The associated long exact sequence is of the form
\[
0 \to H^{0}(\Omega_{\mathbb{C}/\mathbb{Q}}^{1}\otimes_{\mathbb{C}} O_{X}) \to \cdots \to H^{2}(\Omega_{X/\mathbb{Q}}^{1}) \xrightarrow{\delta r} H^{2}(\Omega_{X/\mathbb{C}}^{1}) \xrightarrow{f} H^{3}(\Omega_{\mathbb{C}/\mathbb{Q}}^{1}\otimes_{\mathbb{C}} O_{X}) \to \cdots.
\]
So the image of  $\delta r$ 
can be identified with the kernel of $f$,
\[
\mathrm{Im}(\delta r)=\mathrm{Ker}(f).
\]

The dual of  $H^{2}(\Omega_{X/\mathbb{C}}^{1}) \xrightarrow{f} H^{3}(\Omega_{\mathbb{C}/\mathbb{Q}}^{1}\otimes_{\mathbb{C}} O_{X}) \ (\cong \Omega_{\mathbb{C}/\mathbb{Q}}^{1}\otimes H^{3}(O_{X}))$,
can be rewritten using the Poincar\'e residue representation as 
\[
\Omega_{\mathbb{C}/\mathbb{Q}}^{1 \ \ast}\otimes R^{d-5} \to R^{2d-5},
\] 
where $R^{j}$ is the Jacobian ring of the hypersurface at degree $j$ (The assumption ``$d \geq 6$" guarantees $d-5 \geq 1$).

Let $S^{d} \subset \mathbb{C}[z_{0}, z_{1}, z_{2}, z_{3}]$ denote homogeneous polynomials of degree $d$. The map $\Omega_{\mathbb{C}/\mathbb{Q}}^{1 \ \ast} \to R^{d}$ is defined as 
\[
\dfrac{\partial}{\partial \alpha}  \to \dfrac{\partial F}{\partial \alpha},
\]
where $\alpha \in \mathbb{C}$ and $F$ is the equation of the hypersurface $X$. Let $\alpha$ run through all the complex numbers, then $\dfrac{\partial F}{\partial \alpha}$ generates a subspace of  $S^{d}$, denoted $W$. 

Now, the map
\[
\Omega_{\mathbb{C}/\mathbb{Q}}^{1 \ \ast}\otimes R^{d-5} \to R^{2d-5}
\] 
can be described as a composition
\[
\Omega_{\mathbb{C}/\mathbb{Q}}^{1 \ \ast}\otimes R^{d-5} \to  W   \otimes R^{d-5}  \to R^{2d-5}, 
\]
where the right map $ W \otimes R^{d-5}  \to R^{2d-5}$ is polynomial multiplication.

Since $X$ is general, all the coefficients of $F$ are algebraically independent, 
the codimension of $W$ in $S^{d}$ is zero. By taking $p=0$ and $k=2d-5$
in the Theorem on page 74 of \cite{Green-note} 
 we see that $ W \otimes S^{d-5}  \to S^{2d-5}$ is surjective. Consequently,  
$ W \otimes R^{d-5}  \to R^{2d-5}$ is surjective because of the following commutative diagram (both of the vertical arrows are surjective),
\[
\begin{CD}
       W \otimes S^{d-5}  @>>>  S^{2d-5} \\
     @VVV   @VVV \\
    W \otimes R^{d-5}  @>>>  R^{2d-5} .
  \end{CD}
\]

So the map  $ \Omega_{\mathbb{C}/\mathbb{Q}}^{1 \ \ast}\otimes R^{d-5} \to R^{2d-5}$ is surjective. Dually, the map $H^{2}(\Omega_{X/\mathbb{C}}^{1}) \xrightarrow{f} H^{3}(\Omega_{\mathbb{C}/\mathbb{Q}}^{1}\otimes_{\mathbb{C}} O_{X})$ is injective. In conclusion, $\mathrm{Im}(\delta r)=\mathrm{Ker}(f)=0$.

\end{proof}

\textbf{Acknowledgements} 
 The author is truly grateful to Jerome Hoffman for discussions and comments that improve this note a lot. This work is partially supported by the Fundamental Research Funds for the Central Universities.

\end{document}